\documentclass[twoside,a4paper,reqno,11pt]{amsart} 
\usepackage[top=28mm,right=28mm,bottom=28mm,left=28mm]{geometry}
\usepackage{microtype}
\usepackage{amsfonts, amsmath, amssymb, amsbsy, mathrsfs, bm, latexsym, stmaryrd, hyperref, array, enumitem, textcomp, url}
\usepackage[table]{xcolor}

\newcommand{\leqs}{\leqslant}
\newcommand{\geqs}{\geqslant}

\newcommand{\PGL}{\operatorname{PGL}}
\newcommand{\GL}{\operatorname{GL}}
\newcommand{\SL}{\operatorname{SL}}
\newcommand{\PSL}{\operatorname{PSL}}
\newcommand{\Sp}{\operatorname{Sp}} 
\newcommand{\SO}{\operatorname{SO}} 
\newcommand{\OO}{\operatorname{O}} 
\newcommand{\SU}{\operatorname{SU}} 

\newcommand{\FF}{\mathbb{F}}

\makeatletter
\newcommand{\imod}[1]{\allowbreak\mkern4mu({\operator@font mod}\,\,#1)}
\makeatother

\renewcommand{\mod}[1]{\imod{#1}}
\renewcommand{\leq}{\leqs}
\renewcommand{\geq}{\geqs}

\theoremstyle{plain}
\newtheorem{theorem}{Theorem} 
 
\newtheorem{corol}[theorem]{Corollary}
\newtheorem{thm}{Theorem}[section] 
\newtheorem{lem}[thm]{Lemma}

\newtheorem{cor}[thm]{Corollary} 

\newtheorem*{theorem*}{Theorem} 
\newtheorem*{conj*}{Conjecture}

\theoremstyle{definition}

\begin{document}

\title{Commuting Involutions and Elementary Abelian Subgroups of Simple Groups}

\author{Robert M. Guralnick}
\address{R.M. Guralnick, Department of Mathematics, University of Southern California, Los Angeles, CA 90089-2532, USA}
\email{guralnic@usc.edu}

\author{Geoffrey R. Robinson}
\address{G. R. Robinson, Department of Mathematics, King's College, 
Aberdeen,  AB24 3FX, UK}
\email{g.r.robinson@abdn.ac.uk}

\dedicatory{Dedicated to the memory of our good friend Jan Saxl}

\date{\today} 

\thanks{The first  author was partially supported by the NSF
grant DMS-1901595 and a Simons Foundation Fellowship 609771} 

\thanks{We thank Thomas Breuer for computations with sporadic groups and 
Ross Lawther for comments on the exceptional groups in characteristic $2$.}

\begin{abstract}    Motivated in part by representation theoretic questions, we prove that if $G$ is a finite quasi-simple group, then there exists
an elementary abelian subgroup of $G$ that intersects every conjugacy class
of involutions of $G$.  

\end{abstract}

\maketitle

\section{Introduction} \label{s:intro}   Let $G$ be a finite group.  An elementary abelian subgroup $E$ of
$G$ is a called a broad subgroup if every involution of $G$ is conjugate to an element of $E$. This definition is motivated in part
by the fact that (using a Theorem of R.Kn\"orr \cite{K}) an irreducible character which vanishes identically on non-identity elements of $E$ lies 
in a $2$-block of defect zero. This in turn allows us to prove that when $G$ contains a broad subgroup, the number of irreducible characters of $G$ 
which lie in $2$-blocks of positive defect of $G$ is at most $|C_{G}(t)|$ for some (non-trivial) involution $t \in G$. Our main result (which depends upon
the classification of finite simple groups)  is:  

\begin{theorem} \label{t:main}  Let $G$ be a finite quasi-simple group.   Then $G$ contains a broad subgroup.
\end{theorem}

Notice that it is immediate that a direct product of finite groups has a broad subgroup if each direct factor has a broad subgroup. 

We have the following corollary, which may be of independent interest:

\begin{corol}  Let $G$ be a finite quasi-simple group.  If $x$ and $y$ are involutions of
$G$, then $x$ commutes with some conjugate of $y$.
\end{corol}

The promised application to character theory is provided by:

\begin{corol} \label{c:positive} 
 Let $G$ be a finite group with no non-trivial normal subgroup of odd order. Then there is a non-trivial involution $t \in F^{\ast}(G)$ such that the total number of irreducible characters of $G$ which do not lie over $2$-blocks of defect zero of $F^{\ast}(G)$ is at most $|C_{G}(t)|$.
\end{corol}	

This applies in particular when $G$ is a non-Abelian simple group, and since $G = {\rm SL}(2,2^n)$ (for $n \geq 2)$ has exactly $2^{n}$ irreducible characters which lie in $2$-blocks of positive defect, 
while $|C_{G}(t)| = 2^{n}$ for each non-trivial involution $t \in G$, the inequality of Corollary 3 may be sharp.  
If $G$ is quasisimple and the center has even order, then there is a central element of order $2$ and so the corollary
is not useful.  

Note that there is no direct analog of the theorem for odd primes.   For example if $p \ge 5$, let $S=\SL_n(p)$ with 
$4 \le n \le  p$.   Since $Z(S)$ has order prime to $p$, it suffices to produce an example in $S$.  
  Let $x$ be a regular unipotent element (i.e. $x$ has a single Jordan block of size $n$).    Note that 
  $x$ has order $p$.  
Then the centralizer of $x$ is an elementary abelian $p$-group of rank $n-1$.  It is straightforward
to check that the only elements $y \in C_S(x)$ which have two Jordan blocks have Jordan blocks
of sizes $(n \pm 1)/2$ if $n$ is odd and  two Jordan blocks of size $n/2$ if $n$ is even.  Thus, $x$ does
not commute with an element of order $p$ that has Jordan blocks of size $n-1$ and $1$.   

 For $p=3$, we give an example where elements of order $3$ are semisimple.  Let $S=\SL_9(q)$
with $q \equiv 4 \mod 9$.    Then $Z= Z(S) = \langle z \rangle$ has order $3$.   Let $x \in S$ 
satisfy $x^3=z$.   So in $G=S/Z$,  $xZ$ is an element of order $3$.  Let $y \in S$ be an element
of order $3$ with a $7$ dimensional eigenspace with eigenvalue $1$.  We claim that
no conjugate of $y$ commutes with $x$ modulo $Z$.    Suppose there is a conjugate $y'$ of $y$
commuting with $x$ modulo $Z$.   Then replacing $y'$ by $y$, we see that  $[x,y]  = u  \in Z$.   Then
$y$ is conjugate to $yu$ whence by considering eigenvalues, $u=1$.  So $x$ and $y$ commute
in $S$.  Since $y$ has a $1$-dimensional eigenspace,  $x$ must preserve this space.  However, 
$x$ has no eigenvalues in the base field, a contradiction.        

The result also fails in general for almost simple groups for $p=2$.  
Let $G = \OO^+_{4m}(2^a)$.   There are two conjugacy classes of involutions in $\SO_{4m}^+(2^a)= \Omega_{4m}^+(2^a)$
interchanged by an involution in $G$.   Clearly such an outer involution cannot commute with
any element of either of the two conjugacy classes interchanged.   

Note that all three families of counterexamples above are also counterexamples to the corollary.  
We also remark that it is easy to check that 
the quaternion group $Q_{8}$ is the only extra-special $2$-group which has a broad subgroup, while it is clear that for $p$ odd, 
a non-Abelian $p$-group of exponent $p$ never has a broad elementary $p$-subgroup.

In the next three sections, we prove Theorem \ref{t:main} for the various families of quasi-simple groups.  In the last two sections,
we prove Corollary \ref{c:positive} and discuss some other character theoretic results.

\section{Alternating and Sporadic Groups}

We make a few preliminary remarks before starting the proof.   We may assume that $G$ has no odd order
normal subgroups and in particular if $G$ is quasisimple, we may assume that the center is a $2$-group.

If $G$ has a unique class of noncentral involutions, then the result holds by considering the  abelian
subgroup $E = \langle Z(G), x \rangle$ where $x$ is a noncentral involution.   If $G$ has $3$ conjugacy classes
of involutions, $C_1, C_2$ and $C_3$ and there exist $x_i \in C_i$ with $x_1x_2x_3=1$, then 
the elementary abelian subgroup $\langle x_1, x_2 \rangle$ intersects each $C_i$.   If $H$ is a subgroup of $G$
and $H$ intersects every conjugacy class of involutions of $G$ and $H$ has a broad subgroup, then so does $G$.  

The case of alternating groups is trivial. 

\begin{lem}  Let $G=A_n, n \ge 5$.  Let $E$ be the  elementary abelian subgroup  of $S_n$
with all orbits of size $2$ if $n$ is even and all orbits but one of size $2$ if
$n$ is odd.    Then every involution in $S_n$ is conjugate to an element of $E$ 
by an element of $A_n$.
\end{lem} 

\begin{lem}  Let $G=2A_n, n \ge 5$.   Then $G$ contains a broad subgroup.
\end{lem}

\begin{proof}  If $x \in A_n$ is an involution, then $x$ lifts to an involution in $G$ if and only 
if the number of points moved by $x$ is a multiple of $8$.   Now choose a partition of 
$n$ into subsets  $X_0, X_1, \ldots, X_d$ where $|X_i|=8$ for $i > 0$ and $|X_0| < 8$.    Let $E$ be the elementary
abelian $2$-group  $E_1 \times \ldots \times E_d$ where $E_i$ acts as a regular elementary
abelian group of order $8$ on $X_i$.   Note that if $x \in E$, then $x$ moves $8e$  points for some
$0 \le e \le d$ and so every element in $E$ lifts to an involution in $G$, whence the lift of $E$ to $G$
is also elementary abelian.
\end{proof}

All but five of the sproadic groups have at most $2$ noncentral conjugacy classes of 
involutions and so the observations above yield the result.  
The remaining cases are  $Co_2, Co_1, Fi_{22}, Fi_{23}$ and $BM$.   The first four have three nontrivial classes 
of involutions and $BM$ has four such classes.   In these cases, we just use GAP \cite{B}  to produce
$E$.   One deals with the quasi-simple sporadic groups in a similar manner \cite{B}.  

\begin{lem}  Let $G$ be a quasi-simple sporadic group.  Then there exists an elementary abelian
subgroup $E$ intersecting every conjugacy class of involutions in $G$.  
\end{lem} 

\section{Group of Lie type in characteristic $2$}

In this section $q$ is a power of $2$.   We consider the finite simple groups of Lie type over $\FF_q$. 
We will typically work with the simply connected group (which will have center of odd order) which
is sufficient.    We note the result that centralizer of an involution is connected.  
It is not true that all centralizers of unipotent elements are connected.

\begin{lem}  Let $G$ be a connected reductive algebraic group over any algebraically closed field
of characteristic $2$.   Let $g \in G$ be an involution.   Then $C_G(g)$ is connected.
\end{lem}   

\begin{proof}   Clearly, we can assume that $G$ is semisimple.  Then $Z(G)$ is finite of odd order
and so the result does not depend upon the isogeny class of $G$.  So we may
assume that $G$ is a direct product and so it suffices to assume that $G$ is simple.   

If $G$ is exceptional, we just quote \cite[Chapter 22]{LS}.   So assume that $G$ is classical.
By Lang's theorem, it suffices to show that if two involutions in $G(q)$  (untwisted) are conjugate in
the algebraic group, they are already conjugate in $G(q)$ (and indeed we may assume that $q$
is large).    If $G=\SL_{2n}(q)$, then any involution is conjugate to one in the unipotent radical $Q$
of the stabilizer of $n$-dimensional space and two elements in $Q$ are conjugate if and only
if they have the same rank (identify $Q$ with $n \times n$ matrices), whence the result.   If
$G=\SL_n(q)$, we can reduce to the case that $n$ is even (or use a slightly different unipotent
radical).   Similarly if $G = \Sp_{2n}(q)$, then any involution acts trivially on a totally singular 
$n$-dimensional subspace \cite{FGS} and so is contained in the unipotent radical $Q$ of the stabilizer of
such a space \cite{FGS}.    We can identify $Q$ with the space of symmetric  bilinear forms on an $n$-dimensional 
space and the orbits of the Levi subgroup are the congruence classes of such forms.  Two forms
are congruent if and only they have the same rank and either both are alternating or not and again
the result holds.

Next consider $G = \SO^+_{2n}(q)=\SO(V)$.   Let ${\bar G}$ be the corresponding algebraic group.
Let $C$ be a $G$-conjugacy class of involutions and let ${\bar C}$ be the conjugacy class
of $\bar{G}$ containing $C$.     
Recall that $G \le H = \Sp_{2n}(q)$.   It follows by  \cite{AS} or \cite{FGS} that if $D$ is the $H$-class containing
$C$, then $D \cap G = C$ or $C \cup \tau(C)$ for $\tau$ any element of $\OO_{2n}^+(q)$ not in $\SO_{2n}^+(q)$.   By
the result for symplectic groups, $D$ is precisely the set of $\FF_q$ points of the corresponding
conjugacy class in the algebraic group.  

Let $x \in C$.   Let $( \ , \  )$ denote the alternating
form preserved by $G$.      If $(xv,v) \ne 0$ for some 
$v \in V$, then $x$ preserves a nondegenerate $2$-space and so $x$ commutes with a transvection $\tau$
(over the base field).  In particular, $\tau(C) = C$ and so $C = D \cap G =\bar{C} \cap H$ as required.
     Suppose that $(xv,v)=0$ for all $v \in V$.   It follows
just as for the symplectic case 
that $x$ acts trivially on a totally singular $n$-dimensional space and so $x$ is contained in the
unipotent radical of the stabilizer of such a space and the orbits of the Levi subgroup on the unipotent
radical correspond to the orbits of $\GL_n$ on skew symmetric matrices.  The only invariant is the rank
and so there is no fusion after field extensions.   Thus $\bar{C} \cap H = C$ as required.  

\end{proof}  

\subsection{Linear, Unitary and Symplectic Groups}

The case of $\SL_n(q)$ is obvious. 

\begin{lem}  Let  $G=\SL_n(q), n \ge 2$.   Let $m=n/2$ if $n$ is even and let $m=(n+1)/2$ if $n$ is odd.
Let $P$ be the stabilizer of an $m$-dimensional subspace $W$.   Then any involution in $G$
is conjugate to an element acting trivially on $W$.   In particular, if $Q$ is the unipotent radical
of the stabilizer of $W$, then $Q$ is elementary abelian and every involution of $G$ is conjugate to
an element of $Q$.  
\end{lem}

For symplectic and unitary groups, we use the following elementary observation.

\begin{lem}  Let $G=\SU_{2n}(q), n \ge 2$  or $\Sp_{2n}(q), n > 2$.
Then any involution $g$ acts trivially on a totally isotropic subspace of dimension $n$
and so the elementary abelian subgroup $Q$ that is the unipotent radical of the stabilizer of
a totally singular subspace of dimension $n$ intersects every conjugacy class of involutions. 
\end{lem} 

\begin{proof}    For $G=\Sp_{2n}(q)$,  the argument is easier (and is given in \cite{FGS}).
So we assume that $G=\SU_{2n}(q)$.  If $n=1$, this is clear.

Let $W$ be the fixed space of the involution of $g$.  
Note that $\dim W  \ge n$.   If $W$ is totally singular, the result follows.  
Otherwise $g$ is trivial on a nonsingular $1$-space $L$.
Then $g$ acts on $L^{\perp}$ and since
$\dim L^{\perp} = 2n-1$,  the fixed space of $g$ has dimension at least $n$ and so is not totally singular. 
Thus 
$g$ leaves invariant some nondegenerate $1$-space orthogonal to $L$ and so
acts trivially on a nondegenerate $2$-space.    The result then follows by induction.  
\end{proof}  

If $G=\SU_{2n+1}(q)$, then any involution is conjugate to an element of 
$\SU_{2n}(q)$.  So to prove the theorem  it suffices to consider only even dimensional
unitary groups.  Thus, the previous lemma gives the result for $\SU$ and $\Sp$. 
 
\subsection{Orthogonal Groups} 

We now consider $G=\SO_{2n}^{\epsilon}(q)=\SO(V)$ with $n \ge 4$.   This case is different than
the other families of classical groups and we have to examine the conjugacy classes of involutions
more closely.    We should note that (in characteristic $2$)  we denote the simple group by $\SO$ (so in particular
$\SO$ is not the full orthogonal group of determinant $1$ isometries preserving the quadratic 
form  -- often $SO^{\epsilon}_n(q)$  is denoted by $\Omega^{\epsilon}_n(q)$).  

Suppose that $n = 2m $ is even.   Decompose $V = V_1 \perp \ldots \perp V_m$ as the orthogonal sum of $m$ four dimensional
spaces.  If $V$ has $+$ type, choose all the summands to be of $+$ type.   Otherwise choose
$V_m$  to be of $-$ type.     Let $E$ be a Sylow $2$-subgroup of $\SO(V_1) \times \ldots \times \SO(V_m)$.  Note
that $E$ is elementary abelian and that the normalizer of $E$ contains elements in $\OO(V)$ not in $\SO(V)$.
Thus  it suffices to show that $E$ intersects every $\OO(V)$ conjugacy class of involutions.  Such involutions
are described in \cite[Chap. 6]{LS}.   They possible conjugacy classes are labelled
 $W(2)^e \oplus W(1)^f$ or $V(2)^2 \oplus W(2)^e \oplus W(1)^f$  (see \cite[Chapter 6]{LS}).   
 They are distinguished by their Jordan form and whether $(gv,v)$ vanishes identically for $v \in V$ with respect
 to the alternating form left invariant by $G$ (the classes involving $V(2)^2$ are the ones where this function does
 not vanish everywhere).  

Note that involutions in $\SO_4^+(V)$ are either trivial,  $W(2)$ or $V(2)^2$.   Involutions in $\SO_4^-(V)$ are 
either trivial or are in the class $V(2)^2$.     If $V$ has $+$ type, then it is clear
any class of involutions intersects $E$.   Similarly if  $V$ has $-$ type,  it is clear that any involution other than $W(2)^m$
is conjugate to an element of $E$.  However, this class of  involutions in the algebraic group is not invariant under the graph
automorphism (i.e. there are two $\SO$ classes which are fused in $\OO$).   Since the classes are stable under any
field automorphism, it follows that this class is not present in $\SO^-_n(q)$.    

If $n$ is odd,  decompose $V = W \perp U$ with $\dim U = 2$ and $W$ of $+$ type (and so
$U$ as the same type as $V$).  
By \cite[Chap. 6]{LS},  any involution in $g$ is conjugate to an element in $\SO(W)$ and so  the result
follows by the case that $n$ is even.    This completes the proof of the theorem in this case.

\subsection{Exceptional Groups}

Note that for the groups  ${^2}B_2(2^{2a+1})$, ${^2}F_4(2^{2a+1})'$ or $G_2(q)$, there are at most
$2$ noncentral conjugacy classes of involutions and the theorem holds -- see \cite[Chap. 22]{LS}.   
Also since centralizers of involutions are connected,  all classes of involutions intersect a subfield
group and so one can work over the base field (and so use the character tables).   We do not use this.

In $F_4(q)$, the center of a Sylow $2$-subgroup has order $q^2$ and intersects of three of the four
nontrivial classes of involutions and so the result holds in the case as well.
It is easy to see that each conjugacy class of involutions in ${^2}E_6(q)$ or $E_6(q)$ intersects $F_4(q)$.
This can be seen by the tables in \cite{LS}  -- indeed they exhibit the classes of unipotent elements 
in the exceptional groups by first working in $E_8$ and then determining which classes intersect
the smaller groups and if they split.  Alternatively, using the Jordan block structure of involutions of
$F_4$ acting on the Lie algebra and the $26$ dimensional module \cite{Law} and noting that the 
Lie algebra of $E_6$ as an $F_4$-module  is the direct sum of the Lie algebra of $F_4$ and the $26$-dimensional module, one can see there
three of the nontrivial classes in $F_4$ remain distinct in $E_6$.  Thus, the result for these cases follows from
the result for $F_4$.

In $E_7$, using the standard labelling, we see that we can choose $4$ commuting simple root subgroups
and the group generate by meets all four classes (see the labelling given in \cite{AS} or \cite{Law}).   Each
class of involutions in $E_8$ intersects $E_7$ (arguing as in the case of $F_4 < E_6$).   Thus, the theorem
holds for $E_7$ and $E_8$. 

\subsection{Exceptional Multipliers}  

If $G$ is a simple group of Lie type in characteristic $2$, then almost always its Schur mulitplier has odd order
and so there is nothing more to do.  There are a handful of cases where this fails.  See \cite[Table 6.1.3]{GLS}.
In each of the cases, one produces the required broad subgroup using GAP \cite{B}.   

\section{Groups of Lie type in odd characteristic}

In this section, we consider finite simple groups of Lie type over a field $\FF_q$ with $q$ odd.  We prove 
a slightly stronger result which makes the proof easier.   So let $X$ be a simple algebraic group over
an algebraically closed field of odd characteristic $p$.   Let $\sigma$ be a Lang-Steinberg endomorphism
of $X$ and let $H=X_{\sigma}$ the fixed points of $\sigma$.  Then $H$ is a finite group of Lie type over
a finite field $\FF_q$ for some $q$ a power of $p$.    Let $G = O^{p'}(H)$.   Then aside from the cases
that $H={^2}G_2(3)$ or $H$ is $\SL_2(3)$ or $\PGL_2(3)$, we have $G$ is quasi-simple.   Note that in the last cases $G$
is solvable and in the first case $[H,H] \cong \PSL_2(8)$.  We will exclude these cases (clearly the main theorem holds
in these cases).    

We say that a subgroup $S$ of $H$ is toral if $S$ is contained in a torus $T$ in $X$.   It follows that $S$ is contained
in a maximal torus of $H$ (which is defined to be $C_T(\sigma)$ for $T$ a $\sigma$ invariant maximal torus of $X$).
Note that if $\phi: X \rightarrow Y$ is an isogeny of algebraic groups, then tori map to tori and the inverse image of a maximal
torus is a maximal torus. 
 
We shall prove the following:

\begin{thm} \label{t:toralbroad}  Let $G$ be a finite quasi-simple group of Lie type over a field of odd characteristic.   Then there exists a toral
subgroup $E$ of $G$ which intersects every conjugacy class of involutions of $G$.
\end{thm} 

If $G$ is as in the theorem and $Z$ is a central subgroup,  then the result for $G/Z$ implies the result for $G$
(since as noted above toral subgroups lift to toral subgroups and any involution in $G$ maps to an involution in
$G$).   On the other hand, if we prove that there exists a toral subgroup $S$ of $G$ such that $S$ intersects
every conjugacy class of $2$-elements $g \in G$ with $g^2 \in Z$, then $SZ/Z$ is toral and intersects each
class of involutions in $SZ/Z$.    Thus, we can choose
a particular form of the group and prove the result needed for that form.

If $G$ is simply connected and split, the result is quite easy.

\begin{lem}  \label{l:splitsc} Let $G$ be a finite quasi-simple simply connected group of Lie type over the field of 
$q$-elements.  
Let $S$ be a maximal torus of $G$ contained in a Borel subgroup of $G$.  Then $S$ intersects every conjugacy class
of involutions of $G$.
\end{lem}

\begin{proof}  Let $T$ be a maximal torus of $X$ containing $S$.   Every involution of $X$ is conjugate to an element
of $T$ (indeed every semisimple element of $X$ is conjugate to an element of $T$).    Note that $S$ is a direct
product of $r$ copies of a cyclic group of order $q-1$ and so that $S$ contains the $2$-torsion subgroup of $T$.
Let $g \in G$ be an involution, then $g$ is conjugate in $X$ to some element of $S$.  Since $X$ is simply connected, 
$g^X \cap G = g^G$ and so $g$ is conjugate in $G$ to an element of $S$.
\end{proof}

We can now complete the proof for the exceptional groups.   The previous lemma implies the result for
the simply connected groups of type $G_2(q)$, $F_4(q)$, $E_6(q)$, $E_7(q)$ and $E_8(q)$.   Aside
from the case of $E_7(q)$, the centers are either trivial or have order $3$ and so the result holds.    

We  note that in the simply connected group $E_7(q)$, there are maximal tori   that are direct products
of seven cyclic groups of order $q-1$ and also of order $q+1$ \cite{LSS}.   Arguing as above, one of these maximal
tori contain a conjugate of any element of order $4$ in $G$ and so any element whose square is central.

This leaves only the cases of ${^2}G_2(3^{2k+1}), k \ge 1$ and ${^2}E_6(q)$.  In the first case, there is only
one class of involutions and the result follows since any semisimple element is contained in some maximal
torus.  In the case of ${^2}E_6(q)$ (and since the center has odd order, it suffices to consider the
simply connected case), there is a maximal torus that is the direct product of six copies of the cyclic
group of order $q+1$ \cite{LSS}.  This contains the $2$-torsion subgroup of a maximal torus of the algebraic group 
and we argue as above.  

We next consider the classical groups.  We will work with the form of the group that acts on the natural module $V$.  

We first consider $\SL_n(q)$.  If $n$ is odd, the center has odd order and so the result follows by Lemma \ref{l:splitsc}.  
We next handle the case that $n$ is even. 

\begin{lem}  Let $G = \SL_{2m}(q)  \le GL_{2m}(q) = L $.     Then there exists a toral subgroup $A$ of 
$L$ such that $A$ intersects every conjugacy class of elements of $G$ containing
elements whose square is central. 
\end{lem}

\begin{proof}  Note that semisimple elements in $G$ which are conjugate in $L$ are already
conjugate in $G$.  Let $Z$ be the Sylow $2$-subgroup of
the center of $G$ with $z$ a generator.    Consider the subgroup $B = B_1 \times \ldots \times B_m$ of $L$ with each $B_i \cong \GL_2(q)$
acting on a direct sum of $m$ two dimensional spaces.  Let $A_i$ be a cyclic subgroup of $B_i$ of order $q^2-1$
and set $A=A_1 \times \ldots \times A_m$.   Note that $A$ contains an element $w$ with $w^2=z$ (by considering
the two dimensional case).   Suppose that $y^2 = z^i$ with $y \in G$.    If $i$ is odd, then there is an odd power of $y$ with square $z$
and so $y$ is conjugate to a power of $w$.  If $i=2j$, then $yz^{-j}$ is an involution in $G$.  It is easy to see that an involution
in $G$ is conjugate to an element of $A$ (since each eigenspace is even dimensional).  Thus,  $A$ contains a conjugate of
every element of $G$ whose square is central.  Clearly $A$ is toral and so the result holds.
\end{proof}

As we noted above, this implies the result for any quotient of $\SL$.  

The same proof holds for $G=\SU_{2m}(q), m \ge 2$ and so the main result for any quotient of $\SU_{2m}(q)$.
  If $G=\SU_n(q)$ with $n$ odd,  then $Z(G)$ has odd order and we see that there
is a toral subgroup that is a direct product of $n-1$ copies of a cyclic group of order $q+1$ which contains the $2$-torsion
subgroup of a maximal torus of the algebraic group and the result follows as above.

If $G = \Sp_{2n}(q)$ with $q \equiv 1 \mod 4$, then the split torus $T$ contains the $4$-torsion
subgroup of the maximal torus containing it in the algebraic group.   Thus every element 
of order $4$ in $G$ is conjugate to an element of $T$ whence the results holds for $\Sp_{2n}(q)$
and its central quotient.   If $q \equiv 3 \mod 4$, there is a maximal torus of $G$ that is the direct
product of $n$ copies of the cyclic group of order $q+1$ which again contains the $4$-torsion subgroup
of a maximal subgroup of the algebraic group and the result follows.  

Finally we consider the quasi-simple groups related to orthogonal groups in dimension at least $7$.  We will prove the
result for the groups $G=\Omega^{\pm}_n(q) \le L = \SO^{\pm}_n(q)$.  As noted above, this suffices
to prove the result in general.

First suppose that $n = 2m +1$ is odd. 
Then  $\SO_n(q)=\SO(V)$ has no center and the simple group 
$\Omega_n(q)$ consists of the elements with spinor norm $1$ and  has index $2$.     Decompose $V = W \perp L$ with
$W$ a hyperplane such that the central element in $\SO(W)$ has spinor norm $1$.   Now decompose
$W = W_1  \perp \ldots \perp W_m$ as an orthogonal sum of $m$ nondegenerate $2$-paces with at least one summand of each type.    
If $g$ is an  involution in $\Omega_n(q)$,  then $g$ is conjugate to an element of the maximal torus $T =
\SO(W_1) \times \ldots \times \SO(W_m)$.    If the multiplicity of the trivial eigenspace of $g$ is greater than $1$, this is clear since
the class of $g$ (for $g$ an involution even in $\SO(V)$) is determined by the dimension and type of its fixed space.
If the fixed space of $g$ is $1$-dimensional and $g$ has spinor norm $1$, then the $-1$ eigenspace of $g$
has the same type as $W$ and so $g$ is conjugate to the element acting as $-1$ on $W$.   Thus, $T \cap \Omega_n(q)$
is a toral  subgroup intersecting each conjugacy class of involutions and the result holds.

Next consider the case that  $n = 4m+ 2, m \ge 2$.    First consider the case that  $L = \SO_n^{\epsilon}(q)= \SO(V)$
with $q \equiv \epsilon 1 \mod 4$,    Note that  $Z(L)=Z(G)$ has order $2$ in this case. 
Decompose $V = V_1 \perp \ldots \perp V_{2m+1}$ where the $V_i$ are $2$-dimensional spaces
of the same type as $V$.     Let $T = \SO(V_1) \times \ldots  \times \SO(V_{2m+1})$.  
Then $T$ is a maximal torus of $\SO(V)$.    It is easy to see that any involution
of spinor norm $1$  is conjugate to an element of $T$.   Moreover, if $x \in G$  and $x^2 = -I$, then
$x$ is conjugate to an element of $T$ (reduce to the two dimensional case).

   Thus,  $T \cap \Omega(V)$ intersects any conjugacy
of elements  of $G$ whose square is central.  This proves the theorem for $\Omega(V)$ and it simple quotient
and so the result holds.

Suppose that $q \equiv  \epsilon 3 \mod 4$.   In this case the nontrivial involution in $Z(L)$ has spinor
norm $-1$ and so $\Omega(V)$ has trivial center and is the simple group.  
  Decompose $V = V_1 \perp \ldots \perp V_{2m+1}$ so that 
there are at least $2$ summands of each type.   Let $T = \SO(V_1) \times \ldots \times \SO(V_{2m+1})$.   
Then $T$ intersects every conjugacy class of involutions of $\SO(V)$ and the result holds.  
 
Finally consider the case that $n=4m, m \ge 2$.    Suppose that $L = \SO_n^-(q) = \SO(V)$.   Note that in
this case $-I$ has spinor norm $-1$ and so $\Omega_n^{-}(q)$ is simple.   Decompose $V = V_1 \perp \ldots \perp V_{2m}$ 
into an orthogonal
sum of $2m$ nondegenerate $2$-dimensional spaces (note both types must occur in this decomposition).   It is clear
that any involution in $\Omega(V)$ is conjugate to an element of $T= \SO(V_1) \times \ldots \times SO(V_{2m})$ and 
the result holds.

Finally suppose that $L=\SO_n^+(q)$.   Decompose $V = V_1 \perp \ldots  \perp V_{2m}$ with each $V_i$ of dimension $2$ and all
of the same type.  If $q \equiv 1 \mod 4$, take the type of $V_i$ to be $+$ and $-$ otherwise.   In the case the  central
involution in $L$ has spinor norm $1$ and indeed the involution in $\SO(V_i)$ has spinor norm $1$.   Set $T = 
SO(V_1) \times \ldots \times SO(V_{2m})$,  a direct product of $2m$ cyclic groups of order $q \pm 1$ (and each
cyclic group has order a multiple of $4$).     If $x \in G$ and
$x^2 = -I$, then $x$ is conjugate to an element of $T$ and any involution of spinor norm $1$ is also conjugate to
an element of $T$ and the result follows.

\section{Proof of Corollary 3}

As mentioned in the introduction, a theorem of R. Kn\"orr \cite{K}, tells us that when $p$ is a prime and $G$ is a finite group of order divisible by $p$, then the irreducible character $\chi$ of $G$ lies in a $p$-block of defect zero of $G$ (hence vanishes on all elements of $G$ of order divisible by $p$) if and only if $\chi$ vanishes on all elements of order $p$ in $G$. Hence when $G$ has an elementary Abelian subgroup $E$ which meets every conjugacy class of elements of order $p$ of $G$, the irreducible character $\chi$ lies in a $p$-block of defect zero of $G$ if and only if $\chi$ vanishes identically on the non-identity elements of $E$.

We may conclude, using the results of \cite{R}  that if $G$ has such an elementary Abelian $p$-subgroup $E$, then $\sum_{ a \in E^{\#}}|\chi(a)|^{2} \geq |E|-1 $ whenever $\chi$ lies in a $p$-block of positive defect of $G$.
	
Using the orthogonality relations, we deduce that if $G$ has $k_{+}$ irreducible characters lying in $p$-blocks of positive defect, we have $$ \sum_{a \in E^{\#}} |C_{G}(a)| \geq k_{+}(|E|-1).$$
In particular, we have $k_{+} \leq |C_{G}(a)|$ for some $a \in E^{\#}.$ 
	
More generally, let $G$ be a finite group, and $N \lhd G$. We will first prove that if $N$ has an elementary Abelian $p$-subgroup $E$ which meets every conjugacy class of elements of order $p$ of $N$, 
then there an element $x$ of order $p$ in $N$ such that at most $|C_{G}(x)|$ irreducible characters of $G$ lie in $p$-blocks which do not cover $p$-blocks of defect zero of $N$.

We first need a slight extension of the above result of Kn\"orr.

\begin{lem} Let $G$ be a finite group and $N$ be a normal subgroup of $G$. Let $B$ be a $p$-block of $G$ with defect group $D$ such that $D \cap N \neq 1.$ Then for any irreducible character $\chi \in B$, there is an element $ x \in N$ of order $p$ with $\chi(x) \neq 0$.
	
\end{lem}

\begin{proof} Let $(R,K,F)$ be a $p$-modular system for $G$. We note that $H = N_{G}(D \cap N) \geq N_{G}(D)$ so by Brauer's First Main Theorem, there is a unique Brauer 
	correspondent block for $B$ in $N_{G}(D \cap N)$, say $B^{\prime}.$ Let $\sigma = \sum_ {\{n \in N: n^{p}=1 \} } n$, which is an element of $Z(RG)$.	
	
	We claim that $\omega_{\chi}(\sigma) \in J(R)$, so that $\omega_{\chi} (\sigma - 1_{G}) \not  \in J(R)$, and in particular, there must be some non-identity element 
	$x \in N$ of order $p$ with $\chi(x) \neq 0$.
	
	Now, using the Brauer homomorphism, there is an irreducible character $\mu \in B^{\prime}$ (which may be assumed to have $D \cap N$ in its kernel) such that 
	$\omega_{\chi}(\sigma) \equiv \omega_{\mu} (\sigma^{\ast} )$ ( mod $J(R)$ ), where $\sigma^{\ast} = \sum_{ \{n \in C_{G}(D \cap N) : n^{p} = 1 \}} n.$
	
	However, let $Z = \Omega_{1}(Z(D \cap N))$, and let $Z^{+}$ be the sum of its elements in $RG$. It is clear that $\sigma^{\ast} = Z^{+}T$ for some element $T$ of $RC_{G}(D \cap N)$ which is itself a sum of certain elements of order dividing $p$. Notice that $T$ commutes with $Z^{+}.$ It follows that $\sigma^{\ast}$ has nilpotent image in $Z(FN_{G}(D \cap N))$, so that $\omega_{\mu} (\sigma^{\ast})  \in J(R)$, which suffices to complete the proof.
	
\end{proof}	

We conclude from the Lemma (and Clifford's Theorem) that the irreducible character $\chi$ of $G$ lies over characters in $p$-blocks of positive defect of $N$ if and only if $\chi(x) \neq 0$ for some element $x$ of order $p$ in $N$.

In particular, if $E$ is the above broad elementary Abelian $p$-subgroup of $N$, then no irreducible character $\chi$ of $G$ which lies over irreducible characters in $p$-blocks of positive defect zero of $N$ can vanish identically on $E^{\#}.$ We may then conclude as above that there is an element $x$ of order $p$ in $N$ such that the total number of irreducible characters of $G$ which do not vanish identically on $p$-singular elements of $N$ is at most $|C_{G}(x)|$.

Finally, we turn our attention to the case $p = 2$, $O(G) = 1$  and $N = F^{\ast}(G)$. If $O_{2}(G) \neq 1$, let $Z = \Omega_{1}(Z(O_{2}(G)) ) \neq 1.$  Then by Clifford's Theorem, whenever $\chi$ is an irreducible character of $G$, ${\rm Res}^{G}_{Z}(\chi)$ is not a multiple of the regular character, so that $\chi$ does not vanish identically on $Z^{\#}.$ Hence, as before, we have $\sum_{ z \in Z^{\#}}|\chi(z)|^{2} \geq |Z|-1$, and we have $k(G) \leq |C_{G}(z)|$ for some $z$ of order $2$ in $Z$, where (as usual), $k(G)$ denotes the number of complex irreducible characters of $G$.

Suppose then that $O_{2}(G) = 1$. Then $F^{\ast}(G)$ is a direct product of non-Abelian simple groups, and $F^{\ast}(G)$ has a broad elementary $2$-subgroup, say $E$,  by Theorem 1 and the remarks following.  Now the arguments above tell us that there is some element $t \in E^{\#}$ such that the number of irreducible characters of $G$ which do not vanish identically on $2$-singular elements of $F^{\ast}(G)$ is at most $|C_{G}(t)|$, and the proof of Corollary 3 is complete.

We remark that if equality holds in Corollary 3, then every irreducible character of $G$  which does not vanish at $t$ takes value $\pm 1$ at $t$, and that each such character has odd degree, so lies in a $2$-block of full defect of $G$ (and has height zero in that block).

\section{Further character-theoretic applications}

One of our motivations for asking the question answered by Corollary 2 is the following:

\begin{lem} Let $G$ be a finite group, and let $t,u$ be non-trivial involutions of $G$ such that $t$ does not commute with any $G$-conjugate of $u$. Let $C = C_{G}(t)$ and $D = C_{G}(u)$.	
	Then there is a non-trivial `irreducible character $\mu$ in the principal $2$-block of $G$ such that $\mu(t)\mu(u) \neq 0$ and $\mu(1) \leq ([C:O(C)]-1)(\sqrt{[D:O(D)]-1})$.
\end{lem}

\begin{proof} We first note that $u$ is not expressible as the product of two conjugates of $t$. More generally, the product of two conjugates of $t$ never lies in the $2$-section of $u$, for if $t^{x}t^{y}$ has $2$-part $u^{z}$, then $t^{x}$ and $t^{y}$ both invert the involution $u^{z}$, contrary to hypothesis.
	
By the usual character-theoretic formula for the coefficient of $g$ in the product of two class sums, we see that the class function $$\sum_{ \chi \in {\rm Irr}(G)} \frac{\chi(t)^{2} \chi}{\chi(1)}	$$ vanishes identically on the $2$-section of $u$ in $G$.

By a well-known consequence of Brauer's Second Main Theorem, the class function $$\sum_{ \chi \in B} \frac{\chi(t)^{2} \chi}{\chi(1)}	$$ also vanishes identically on the $2$-section of $u$ in $G$, where $B$ is the principal $2$-block of $G$. It follows easily that $$\sum_{ 1 \neq \chi \in B} \frac{\chi(t)^{2}|\chi(u)|}{\chi(1)} \geq 1. $$

Choose a non-trivial  irreducible character $\mu \in B$ with $\mu(t)\mu(u) \neq 0$ and with $r(u) =  \frac{|\mu(u)|}{\mu(1)}$ maximal. Then $$\sum_{ 1 \neq \chi \in B} \chi(t)^{2}  \geq \frac{1}{r(u)}.$$ 

By Brauer's Second and Third Main Theorem, we have $$\sum_{ \chi \in B} \chi(t)^{2}	 = \sum_{ \theta \in b} \theta(1)^{2},$$ where $b$ is the principal $2$-block of $C = C_{G}(t)$ .
 Since each irreducible character $\theta$ of $b$ has $O(C)$ in its kernel, we deduce that $$\sum_{ 1 \neq \chi \in B} \chi(t)^{2} \leq [C:O(C)] - 1$$.
 
A similar argument with $u$ certainly allows us to conclude that $|\mu(u)| \leq \sqrt{[D:O(D)]-1}.$ Hence we may conclude that $$\mu(1) \leq ([C:O(C)]-1)( \sqrt{[D:O(D)]-1}),$$ and the result follows.

\end{proof}

In view of Corollary 2, this result may not be applied as it stands to finite simple groups, but we have seen that the hypotheses may be satisfied within almost simple groups, so we mention:

\begin{cor} Let $G$ be an almost simple group with $F^{\ast}(G) = S$, and let $t$ be an involution of $G \backslash S$ and $u$ be an involution of $S$ which commutes with no conjugate of $t$.
	Then the principal $2$-block of $S$ contains a $t$-stable non-trivial irreducible character of degree at most $([C:O(C)]-1)(\sqrt{[D:O(D)]-1}),$ where $C = C_{S}(t)$ and $D = C_{S}(u)$.
	
\end{cor}

\begin{proof} We might as well work within $\langle t \rangle S$, so we suppose that $G = \langle t \rangle S$.	We may argue in a similar fashion to the previous result. Now the principal $2$-block of $G$ contains two linear characters, $1$ and $\lambda$, say, and we also have $\lambda(t)^{2}\lambda(u) = 1.$ This time, we find that $$\sum_{ \chi \in B: \chi(1) >1} \frac{\chi(t)^{2}|\chi(u)|}{\chi(1)}
 \geq 2.$$ where $B$ is again the principal $2$-block of $G$. However, the irreducible characters $\chi \in B$ with $\chi(t) \chi(u) \neq 0$ come in pairs, both members of which lie over the same  $t$-stable irreducible character of the principal $2$-block of $S$ (if $\chi$ is one such, so is $\lambda \chi \neq \chi$, and note that $\chi(t)^{2} \chi(u)$ is unchanged on replacing $\chi$ by $\lambda \chi$). Also, $C_{G}(t) = \langle t \rangle \times  C_{S}(t)$, so the result follows in a fashion similar to the previous Corollary.

\end{proof}

\end{document}